\documentclass{amsart}
\usepackage{amssymb}
\usepackage{amscd}
\usepackage{multicol}
\usepackage{MnSymbol}
\usepackage{tikz}
\usepackage[OT2,T1]{fontenc}
\DeclareSymbolFont{cyrletters}{OT2}{wncyr}{m}{n}
\DeclareMathSymbol{\Sha}{\mathalpha}{cyrletters}{"58}
\title[On Euler's formulae among double zeta values]
{On Euler's formulae for double zeta values \\
}
\author{Ryotaro Harada}

\address{Graduate School of Mathematics, Nagoya University, 
Furo-cho, Chikusa-ku, Nagoya 464-8602 Japan }
\email{m15039r@math.nagoya-u.ac.jp}
\date{October 26, 2016.}

\newtheorem{thm}{Theorem}

\newtheorem{prop}[thm]{Proposition}  

\theoremstyle{remark}
        
\theoremstyle{definition}

\newtheorem{rem}[thm]{Remark}

\begin{document}
\bibliographystyle{amsalpha+}
\maketitle

%%%%%%%%%%%%%%%%%%%%%%%%%%%%%%%%%%%%%%%%%%%%%%%%%%%%%%%%%%%%%%%%%%%%%%
\begin{abstract}
In 1776, L. Euler proposed three methods, called {\it prima methodus}, {\it secunda methodus} and {\it tertia methodus}, to calculate formulae for double zeta values. However strictly speaking, his last two methods are mathematically incomplete and require more precise reformulation and more sophisticated arguments for their justification. In this paper, we reformulate his formulae, give their rigorous proofs and also clarify that the formulae can be derived from the extended double shuffle relations. 
\end{abstract}

%%%%%%%%%%%%%%%%%%%%%%%%%%%%%%%%%%%%%%%%%%%%%%%%%%%%%%%%%%%%%%%%%%%%%%
%\tableofcontents

%To be added..
\section{Euler's methods}
In 1776, L. Euler published a celebrated paper {\it Meditations circa singulare serierum genus}\footnote{Downloadable from {\tt http://eulerarchive.maa.org}.} written in Latin which means 'Meditations about a singular type of series' in English. It is said to be the first publication in history where multiple zeta values (actually only double zeta values) were introduced.
In the paper he proposed three methods to calculate certain relations among double zeta values, which he called {\it prima methodus}, {\it secunda methodus} and {\it tertia methodus}. Here we explain his methods with his idea and point out the steps that would be considered insufficient. 
\subsection{Prima methodus}
%Here we recall the definition of the double zeta value and the double zeta star value 
%\begin{defn}
%For $n\in\mathbb{Z}_{>0}$ and $m\in\mathbb{Z}_{>1}$, we define the {\it double zeta value}
%%\begin{align*}
%\[
%\zeta(n,m):=\sum_{0<k_1<k_2}\frac{1}{{k_1}^{n}{k_2}^{m}},
%\]
%and the {\it double zeta star value}
%\[
%\zeta^\star(n,m):=\sum_{0<k_1\leq k_2}\frac{1}{{k_1}^{n}{k_2}^{m}}.
%\]
%%\end{align*}
%\end{defn}   
In the paper, he studied the series 
%\begin{align*}
%  &1+a(1+\alpha)+b(1+\alpha+\beta)+c(1+\alpha+\beta+\gamma)+\cdots \\
%  &+1+\alpha(1+a)+\beta(1+a+b)+\gamma(1+a+b+c)+\cdots \\
%  &= (1+a+b+c+d+e+\cdots)(1+\alpha+\beta+\gamma+\delta+\epsilon)\\
%  &\quad +(1+a\alpha+b\beta+c\gamma+d\delta+e\epsilon+\cdots).
%\end{align*}
%He considered to use this formula . 
%\begin{align*}
%  &1+\frac{1}{2^m}\biggl( 1+\frac{1}{2^n} \biggr)+\frac{1}{3^m}\biggl( 1+\frac{1}{2^n}+\frac{1}{3^n} \biggr)+\frac{1}{4^m}\biggl( 1+\frac{1}{2^n}+\frac{1}{3^n}+\frac{1}{4^n} \biggr)+\cdots \\
%\end{align*} 
$
1+\frac{1}{2^m}\bigl( 1+\frac{1}{2^n} \bigr)+\frac{1}{3^m}\bigl( 1+\frac{1}{2^n}+\frac{1}{3^n} \bigr)+\frac{1}{4^m}\bigl( 1+\frac{1}{2^n}+\frac{1}{3^n}+\frac{1}{4^n} \bigr)+\cdots
$
which he denoted by the unconventional notation
$\mathit{
  \int\frac{1}{z^m}\Bigl( \frac{1}{y^n} \Bigr)
}$
and also the series 
$
  1+\frac{1}{2^m}+\frac{1}{3^m}+\frac{1}{4^m}+\cdots
$
which he again denoted by the notation
$\mathit{
  \int\frac{1}{z^m}.
}$
In modern language, they are nothing but the {\it double zeta star value}
\[
\zeta^\star(n,m):=\sum_{0<k_1\leq k_2}\frac{1}{{k_1}^{n}{k_2}^{m}}=\zeta(n,m)+\zeta(m+n) 
\]
 for $n\in\mathbb{Z}_{>0}$ and $m\in\mathbb{Z}_{>1}$, and the {\it Riemann zeta value}
$
\zeta(m):=\sum_{k=1}^{\infty}\frac{1}{k^m}
$ 
for $m\in\mathbb{Z}_{>1}$ respectively. 
Here we recall the {\it double zeta value} 
  \[
      \zeta(n,m) := \sum_{0<k_1<k_2}{\frac{1} { {k_1}^{n} {k_2}^{m} } }
  \]
for $n\in\mathbb{Z}_{>0}$ and $m\in\mathbb{Z}_{>1}$. 
By multiplying these series, he obtained the {\bf formula of prima methodus} (cf. \cite{Eu} p.144)
%Then he obtained the following relation from primitive calculation of series. 
%\begin{align*}
%  &1+\frac{1}{2^m}\biggl( 1+\frac{1}{2^n} \biggr)+\frac{1}{3^m}\biggl( 1+\frac{1}{2^n}+\frac{1}{3^n} \biggr)+\frac{1}{4^m}\biggl( 1+\frac{1}{2^n}+\frac{1}{3^n}+\frac{1}{4^n} \biggr)+\cdots \\
%  &+1+\frac{1}{2^n}\biggl( 1+\frac{1}{2^m} \biggr)+\frac{1}{3^n}\biggl( 1+\frac{1}{2^m}+\frac{1}{3^m} \biggr)+\frac{1}{4^n}\biggl( 1+\frac{1}{2^m}+\frac{1}{3^m}+\frac{1}{4^m} \biggr)+\cdots \\
%  &=\int\frac{1}{z^m}\int\frac{1}{z^n}+\int\frac{1}{z^{m+n}}.
%\end{align*} 
%Then he obtained the formula of prima methodus 
\[\mathit{
\int\frac{1}{z^m}\biggl( \frac{1}{y^n} \biggr)+\int\frac{1}{z^n}\biggl( \frac{1}{y^m} \biggr)=\int\frac{1}{z^m}\cdot\int\frac{1}{z^n}+\int\frac{1}{z^{m+n}}.
}\]
In modern language, this says
\begin{equation}
   \label{pristar}  
\zeta^{\star}(n,m)+\zeta^{\star}(m,n)=\zeta(n)\zeta(m)+\zeta(m+n).
\end{equation}
It is nothing but the {\it harmonic product formula},
\begin{align}
  \label{pri}
  \zeta(m,n)+\zeta(n,m)+\zeta(m+n)=\zeta(m)\zeta(n),
\end{align}
which is known to hold for $m,n\in\mathbb{Z}_{>1}$.

\subsection{Secunda methodus}
This subsection summarizes \cite{Eu} pp.144--149. 

Firstly, Euler began with a partial fraction decomposition (cf. \cite{Eu} pp.145--146) 
\begin{align*}
  \mathit{\frac{1}{x^{n}(x+a)^m}= }&\mathit{\frac{1}{a^n}\cdot\frac{1}{x^m}-\frac{n}{1\cdot a^{n+1}}\cdot\frac{1}{x^{m-1}}+\frac{n(n+1)}{1\cdot2a^{n+2}}\cdot\frac{1}{x^{m-2}} }\\
  &\mathit{-\frac{n(n+1)(n+2)}{1\cdot2\cdot3a^{n+3}}\cdot\frac{1}{x^{m-3}}+{\rm etc.} }\\
  &\mathit{\pm\frac{1}{a^m}\cdot\frac{1}{(x+a)^n}\pm\frac{m}{1\cdot a^{m+1}}\cdot\frac{1}{(x+a)^{n-1}}\pm\frac{m(m+1)}{1\cdot 2a^{m+2}}\cdot\frac{1}{(x+a)^{n-2}} }\\
  &\mathit{\pm\frac{m(m+1)(m+2)}{1\cdot 2\cdot 3a^{m+3}}\cdot\frac{1}{(x+a)^{n-3}}\pm{\rm etc.} }
\end{align*}
In modern language it reads
\begin{align}
\label{fra}  
\frac{1}{x^n(x+a)^m} = &\sum^{m-1}_{i=0}(-1)^i\binom{n+i-1}{i}\frac{1}{a^{n+i} }\frac{1}{x^{m-i} }
                                +(-1)^m\sum^{n-1}_{j=0}\binom{m+j-1}{j}\frac{1}{a^{m+j}}\frac{1}{(x+a)^{n-j}}. %\nonumber
\end{align} \par
Secondly, he put\footnote{We note that in \cite{Eu}, it is simply denoted by $s$, which could give rise to confusion in this text.} 
$
  	  s_a:=\sum^{\infty}_{x=1}\frac{1}{x^n(x+a)^m}.
	$
and calculated as follows (cf. \cite{Eu} pp.146--147):
\begin{align*}
  s_a&=\mathit{\frac{1}{a^n}\int\frac{1}{z^m}-\frac{n}{1\cdot a^{n+1}}\int\frac{1}{z^{m-1}}+\frac{n(n+1)}{1\cdot2a^{n+2}}\int\frac{1}{z^{m-2}} }
  \mathit{\ -\frac{n(n+1)(n+2)}{1\cdot2\cdot3a^{n+3}}\int\frac{1}{z^{m-3}}+{\rm etc.} }\\
  &\mathit{\ \pm\frac{1}{a^m}\int\frac{1}{z^n}\pm\frac{m}{1\cdot a^{m+1}}\int\frac{1}{z^{n-1}}\pm\frac{m(m+1)}{1\cdot2a^{m+2}}\int\frac{1}{z^{n-2}} }
  \mathit{\ \pm\frac{m(m+1)(m+2)}{1\cdot2\cdot3 a^{m+3}}\int\frac{1}{z^{n-3}}\pm{\rm etc.} }\\
  &\mathit{\ \mp\frac{1}{a^m}\Bigl( 1+\frac{1}{2^n}+\frac{1}{3^n}+}\cdots\mathit{+\frac{1}{a^n} \Bigr) }
  \mathit{\ \mp\frac{m}{1\cdot a^{m+1}}\Bigl( 1+\frac{1}{2^{n-1}}+\frac{1}{3^{n-1}}+}\cdots\mathit{+\frac{1}{a^{n-1}} \Bigr) }\\
  &\mathit{\ \mp\frac{m(m+1)}{1\cdot2a^{m+2}}\Bigl( 1+\frac{1}{2^{n-2}}+\frac{1}{3^{n-2}}+}\cdots\mathit{+\frac{1}{a^{n-2}} \Bigr) }
  \mathit{\ \mp\frac{m(m+1)(m+2)}{1\cdot2\cdot3a^{m+3}}\Bigl( 1+\frac{1}{2^{n-3}}+\frac{1}{3^{n-3}}+}\cdots\mathit{+\frac{1}{a^{n-3}} \Bigr) }\\
  &\mathit{ \ \mp{\rm etc.} }
\end{align*}
In modern language, this is written as 
\begin{align*}
          s_a=\sum^{\infty}_{x=1}\frac{1}{x^n(x+a)^m}
%=&\sum^{\infty}_{x=1}\biggl\{ \sum^{m-1}_{i=0}(-1)^i\binom{n+i-1}{i}\frac{1}{a^{n+i}}\frac{1}{x^{m-i} }\\
%                                &+(-1)^m\sum^{n-1}_{j=0}\binom{m+j-1}{j}\frac{1}{a^{m+j}}\frac{1}{(x+a)^{n-j} }  \biggr\}\\
                                =&\sum^{m-1}_{i=0}\sum^{\infty}_{x=1}(-1)^i\binom{n+i-1}{i}\frac{1}{a^{n+i} }\frac{1}{x^{m-i} }\\
                                &+(-1)^m\sum^{n-1}_{j=0}\sum^{\infty}_{x=1}\binom{m+j-1}{j}\frac{1}{a^{m+j} }\frac{1}{(x+a)^{n-j} }.    
%	                         =&\sum^{m-1}_{i=0}\binom{n+i-1}{i}\frac{1}{a^{n+i} }\zeta(m-i)\\
%					 &+(-1)^m\sum^{m-1}_{j=0}\binom{m+j-1}{j}\frac{1}{a^{m+j}}\zeta(n-j)
\end{align*}  
\begin{rem}
By using \eqref{fra}, he came to the above equation in the following way:
\begin{align*}
&\sum^{\infty}_{x=1}\frac{1}{x^n(x+a)^m}\\
&=\sum^{\infty}_{x=1}\biggl\{ \sum^{m-1}_{i=0}(-1)^i\binom{n+i-1}{i}\frac{1}{a^{n+i}}\frac{1}{x^{m-i} }
                                +(-1)^m\sum^{n-1}_{j=0}\binom{m+j-1}{j}\frac{1}{a^{m+j}}\frac{1}{(x+a)^{n-j} }  \biggr\}\\
                               &\overset{!}{=}\sum^{m-1}_{i=0}\sum^{\infty}_{x=1}(-1)^i\binom{n+i-1}{i}\frac{1}{a^{n+i} }\frac{1}{x^{m-i} }
                                +(-1)^m\sum^{n-1}_{j=0}\sum^{\infty}_{x=1}\binom{m+j-1}{j}\frac{1}{a^{m+j} }\frac{1}{(x+a)^{n-j} }.   
%	                         =&\sum^{m-1}_{i=0}\binom{n+i-1}{i}\frac{1}{a^{n+i} }\zeta(m-i)\\
%					 &+(-1)^m\sum^{m-1}_{j=0}\binom{m+j-1}{j}\frac{1}{a^{m+j}}\zeta(n-j)
\end{align*}  
 Here we alert reader to the fact that validity of the above equation $\overset{!}{=}$ is really problematic. We can not exchange two summations because the right hand side of this equation does not converge absolutely. 
%Moreover, he considered $\sum^{\infty}_{a=1}s_a$. That is, 
\end{rem}
 Thirdly, he considered $\sum^{\infty}_{a=1}s_a$ and calculated as follows (cf. \cite{Eu} pp.147--148): 
\begin{align*}
 \sum^{\infty}_{a=1}s_a=&\mathit{ \int\frac{1}{z^n}\cdot\int\frac{1}{z^m}-\frac{n}{1}\int\frac{1}{z^{n+1} }\cdot\int\frac{1}{z^{m-1} }+\frac{n(n+1)}{1\cdot2}\int\frac{1}{z^{n+2} }\cdot\int\frac{1}{z^{m-2}} }\\
  &\mathit{\ -\frac{n(n+1)(n+2)}{1\cdot2\cdot3}\int\frac{1}{z^{n+3}}\cdot\int\frac{1}{z^{m-3} }+{\rm etc.}}\\
  &\mathit{\ \pm\int\frac{1}{z^m}\cdot\int\frac{1}{z^n}\mp\int\frac{1}{z^m}\Bigl(\frac{1}{y^n}\Bigr)}
  \mathit{\ \pm\frac{m}{1}\int\frac{1}{z^{m+1}}\cdot\int\frac{1}{z^{n-1}}\mp\frac{m}{1}\int\frac{1}{z^{m+1}}\Bigl(\frac{1}{y^{n-1}}\Bigr)}\\
  &\mathit{\ \pm\frac{m(m+1)}{1\cdot2}\int\frac{1}{z^{m+2}}\cdot\int\frac{1}{z^{n-2} }\mp\frac{m(m+1)}{1\cdot2}\int\frac{1}{z^{m+2} }\Bigl(\frac{1}{y^{n-2} }\Bigr)}\\
  &\mathit{\ \pm\frac{m(m+1)(m+2)}{1\cdot2\cdot3}\int\frac{1}{z^{m+3}}\cdot\int\frac{1}{z^{n-3}} }
\mathit{ \mp\frac{m(m+1)(m+2)}{1\cdot2\cdot3}\int\frac{1}{z^{m+3}}\Bigl(\frac{1}{y^{n-3}}\Bigr) }\\
  &\mathit{\ \pm{\rm etc.} }
\end{align*}  
In modern language, it means
\begin{align}
	\label{key} \zeta(m,n)= &\sum^{\infty}_{a=1}s_a=\sum^{\infty}_{a=1}\sum^{m-1}_{i=0}(-1)^i\binom{n+i-1}{i}\frac{1}{a^{n+i}}\zeta(m-i)\\
       &+(-1)^{m}\sum^{\infty}_{a=1}\sum^{n-1}_{j=0}\binom{m+j-1}{j}\frac{1}{a^{m+j}}\Bigl\{\zeta(n-j)-\sum^{a}_{k=1}\frac{1}{k^{n-j}}\Bigl\}\nonumber\\
       =&\sum^{m-1}_{i=0}(-1)^i\binom{n+i-1}{i}\zeta(n+i)\zeta(m-i)\nonumber\\
					    &+(-1)^m\sum^{n-1}_{j=0}\binom{m+j-1}{j}\Bigl\{ \zeta(m+j)\zeta(n-j)-\zeta^{\star}(n-j, m+j)\Bigr\}. \nonumber
%                                          =&\sum^{m-1}_{i=0}\binom{n+i-1}{i}\zeta(n+i)\zeta(m-i)\nonumber\\
%					    &+(-1)^m\sum^{n-1}_{j=0}\binom{m+j-1}{j}\zeta(m+j, n-j). \nonumber
\end{align}
The above is a {\bf key formula} of secunda methodus and tertia methodus. 
%\begin{rem}
%Here we warn that this formula \eqref{key} contains a meaningless value $\zeta(1)$ which does not exist. We remark that this part will be correctly reformulated in Proposition \ref{propc} (for $P(m,n)$, see \eqref{P(m,n)}). 
%\end{rem}
Finally he substituted the formula \eqref{key} into \eqref{pri} and obtained the {\bf formula of secunda methodus} (cf. \cite{Eu} pp.148--149).
\begin{align*}
&\mathit{\int\frac{1}{z^m}\cdot\int\frac{1}{z^n}-\int\frac{1}{z^{m+n}}}\\
&\mathit{\ =(1\pm1)\int\frac{1}{z^m}\cdot\int\frac{1}{z^n}\mp\int\frac{1}{z^m}\Bigl( \frac{1}{y^n} \Bigr)}\\
&\mathit{\ \ -\frac{m}{1}(1\mp1)\int\frac{1}{z^{m+1}}\cdot\int\frac{1}{z^{n-1}}\mp\frac{m}{1}\int\frac{1}{z^{m+1}}\Bigl(\frac{1}{y^{n-1}}\Bigr)}\\
&\mathit{\ \ +\frac{m(m+1)}{1\cdot2}(1\pm1)\int\frac{1}{z^{m+2}}\cdot\int\frac{1}{z^{n-2}}\mp\frac{m(m+1)}{1\cdot2}\int\frac{1}{z^{m+2}}\Bigl(\frac{1}{y^{n-2}}\Bigr)}\\
&\mathit{\ \ -\frac{m(m+1)(m+2)}{1\cdot2\cdot3}(1\mp1)\int\frac{1}{z^{m+3}}\cdot\int\frac{1}{z^{n-3}}\mp\frac{m(m+1)(m+2)}{1\cdot2\cdot3}\int\frac{1}{z^{m+3}}\Bigl(\frac{1}{y^{n-3}}\Bigr)}\\
&\mathit{\ \ \pm{\rm etc.}}\\
&\mathit{\ \ +(1\pm1)\int\frac{1}{z^n}\cdot\int\frac{1}{z^m}\mp\int\frac{1}{z^n}\Bigl( \frac{1}{y^m} \Bigr)}\\
&\mathit{\ \ -\frac{n}{1}(1\mp1)\int\frac{1}{z^{n+1}}\cdot\int\frac{1}{z^{m-1}}\mp\frac{n}{1}\int\frac{1}{z^{n+1}}\Bigl(\frac{1}{y^{m-1}}\Bigr)}\\
&\mathit{\ \ +\frac{n(n+1)}{1\cdot2}(1\pm1)\int\frac{1}{z^{n+2}}\cdot\int\frac{1}{z^{m-2}}\mp\frac{n(n+1)}{1\cdot2}\int\frac{1}{z^{n+2}}\Bigl(\frac{1}{y^{m-2}}\Bigr)}\\
&\mathit{\ \ -\frac{n(n+1)(n+2)}{1\cdot2\cdot3}(1\mp1)\int\frac{1}{z^{n+3}}\cdot\int\frac{1}{z^{m-3}}\mp\frac{n(n+1)(n+2)}{1\cdot2\cdot3}\int\frac{1}{z^{n+3}}\Bigl(\frac{1}{y^{m-3}}\Bigr)}
\mathit{\pm{\rm etc.} }
\end{align*}
In modern language, it is translated into the following:
\begin{align}
\label{secE}
\zeta(m)\zeta(n)-\zeta(m+n) &=\sum^{{m-1}}_{i=0}(-1)^{i}\binom{n+i-1}{i}\zeta(n+i)\zeta(m-i) \\
  & +(-1)^{m}\sum^{{n-1}}_{j=0}\binom{m+j-1}{j}\Bigl\{ \zeta(m+j)\zeta(n-j)-\zeta^{\star}(n-j, m+j)\Bigr\}\nonumber \\
  & +\sum^{{n-1}}_{i=0}(-1)^{i}\binom{m+i-1}{i}\zeta(m+i)\zeta(n-i)\nonumber \\
  & +(-1)^{n}\sum^{{m-1}}_{j=0}\binom{n+j-1}{j}\Bigl\{ \zeta(n+j)\zeta(m-j)-\zeta^{\star}(m-j, n+j)\Bigr\}\nonumber .
\end{align}
%\begin{rem}
%We warn again that the formula is problematic because it contains a meaningless value $\zeta(1)$ which does not exist. We remark it will be correctly reformulated in our main theorem (Theorem \ref{thma}) as \eqref{eqsec}.  
%\end{rem}
\subsection{Tertia methodus}
Euler proposed another method in \cite{Eu} pp.168--170. He introduced the following unconventional notations in \cite{Eu} pp.165--166:
%\begin{align*}
  $p^{\mu}=p^{\nu}:=\zeta(\mu)\zeta(\nu),\quad 
  p^{\lambda}:=\zeta(\lambda),\quad 
  q^{\mu}:=\zeta^{\star}(\nu ,\mu)=\zeta^{\star}(\lambda-\mu ,\mu)$
%\end{align*}
with $\mu+\nu=\lambda$ and $\mu$, $\nu$ $\in\mathbb{Z}_{>0}$. By using these symbols, he rewrote the equations \eqref{pristar} and \eqref{key} respectively as follows (cf. \cite{Eu} pp.169--170): 
\begin{align*}
   &\mathit{q^m+q^n=p^m+p^{m+n}=p^n+p^{m+n}.}\\
   &\mathit{q^n-p^{m+n}=p^m-\frac{n}{1}p^{m-1}+\frac{n(n+1)}{1\cdot2}p^{m-2}-\frac{n(n+1)(n+2)}{1\cdot2\cdot3}p^{m-3}+{\rm etc}. }\\
                     &\mathit{\qquad \pm q^n\pm\frac{m}{1}q^{n-1}\pm\frac{m(m+1)}{1\cdot2}q^{n-2}\pm\frac{m(m+1)(m+2)}{1\cdot2\cdot3}q^{n-3}\pm {\rm etc}. }\\
  &\mathit{\qquad \mp p^{m+n} \mp \frac{m}{1}p^{m+n}\mp\frac{m(m+1)}{1\cdot2}p^{m+n}\mp\frac{m(m+1)(m+2)}{1\cdot2\cdot3}p^{m+n}\mp {\rm etc}. }
%   \eqref{key}:\ &0=q^m-\frac{n}{1}(q^{m-1}+q^{n+1})+\frac{n(n+1)}{1\cdot2}(q^{m-2}+q^{n+2})\\
%                     &\quad -\frac{n(n+1)(n+2)}{1\cdot2\cdot3}(q^{m-3}+q^{n+3})+{\rm etc}.\\
%  &\qquad \quad +\frac{n}{1}p^{m+n}-\frac{n(n+1)}{1\cdot2}p^{m+n}+\frac{n(n+1)(n+2)}{1\cdot2\cdot3}p^{m+n}-{\rm etc}.\\
%  &\quad \pm q^n\pm\frac{m}{1}q^{n-1}\pm\frac{m(m+1)}{1\cdot2}q^{n-2}\pm\frac{m(m+1)(m+2)}{1\cdot2\cdot3}q^{m-3}\pm {\rm etc}.\\
%  &\quad \mp p^{m+n} \mp \frac{m}{1}p^{m+n}\mp\frac{m(m+1)}{1\cdot2}\mp\frac{m(m+1)(m+2)}{1\cdot2\cdot3}p^{m+n}\mp {\rm etc}.
\end{align*}
By substituting the former equation $p^{\mu} = q^\mu+q^{m+n-\mu}-p^{m+n}$ into the latter equation, he obtained the following (cf. \cite{Eu} p.170)
\begin{align*} 
&\mathit{0=q^m-\frac{n}{1}(q^{m-1}+q^{n+1})+\frac{n(n+1)}{1\cdot2}(q^{m-2}+q^{n+2}) }\\
  &\mathit{\qquad  -\frac{n(n+1)(n+2)}{1\cdot2\cdot3}(q^{m-3}+q^{n+3})+{\rm etc}. }\\
  &\mathit{\qquad +\frac{n}{1}p^{m+n}-\frac{n(n+1)}{1\cdot2}p^{m+n}+\frac{n(n+1)(n+2)}{1\cdot2\cdot3}p^{m+n}-{\rm etc}. }\\
  &\mathit{\qquad \pm q^n\pm\frac{m}{1}q^{n-1}\pm\frac{m(m+1)}{1\cdot2}q^{n-2}\pm\frac{m(m+1)(m+2)}{1\cdot2\cdot3}q^{n-3}\pm {\rm etc}. }\\
  &\mathit{\qquad \mp p^{m+n} \mp \frac{m}{1}p^{m+n}\mp\frac{m(m+1)}{1\cdot2}p^{m+n}\mp\frac{m(m+1)(m+2)}{1\cdot2\cdot3}p^{m+n}\mp {\rm etc}. }
\end{align*}
In modern language, this means that, by \eqref{pristar}, he transformed \eqref{key} into the following:
\begin{align}
\label{terE}
\zeta(m, n)&=\sum^{m-1}_{i=0}(-1)^i\binom{n+i-1}{i}\Bigl\{\zeta(n+i, m-i)+\zeta(m-i, n+i)+\zeta(m+n)\Bigr\}\\
               &\quad +(-1)^m\sum^{n-1}_{j=0}\binom{m+j-1}{j}\zeta(m+j,n-j).\nonumber
\end{align}
This is the {\bf formula of tertia methodus}. 
\begin{rem}
We warn that formulae \eqref{key}, \eqref{secE} and \eqref{terE} contain meaningless values $\zeta(1)$ and $\zeta(m+n-1, 1)$. However, they will be correctly reformulated in Theorem \ref{thma} \eqref{thm7}, Theorem \ref{thmb} \eqref{eqsec} and 
%Theorem \ref{thma} as 
\eqref{eqter}.
%This formula again contains the meaningless values $\zeta(1)$ and $\zeta(m+n-1,1)$ which do not exist. We remark that the formula \eqref{terE} will be correctly reformulated in our main theorem (Theorem \ref{thma}) as \eqref{eqter}.
\end{rem}
\begin{rem}
The sum formula of Granville \cite{Gr} and Zagier (unpublished) in the case of double zeta values is recovered as a special case of \eqref{terE} for $m=1$.  
\end{rem}
\section{Main result}
%We warn that formulae \eqref{key}, \eqref{secE} and \eqref{terE} contain meaningless values $\zeta(1)$ and $\zeta(m+n-1, 1)$. We remark they will be correctly reformulated in Theorem \ref{thma} \eqref{thm7}, Theorem \ref{thmb} as \eqref{eqsec} and \eqref{eqter}.
In this section we give a rigorous reformulation of Euler's problematic formulae and its complete proof in Theorem \ref{thma} and \ref{thmb} by using the generating functions of double zeta values introduced by Gangl, Kaneko, Zagier in \cite{GKZ}.
%\S \ref{reform} and its complete proof in \S \ref{pr}.  

%The multiple zeta value is defined as follows.
%
%
%\label{defna}
%\begin{defn}
%Let $(k_1, \cdots, k_n)$ be an admissible index, that is, $k_1, \cdots, k_n \in \mathbb{Z}_{>0}$ with $k_n>1 $. The {\bf multiple zeta value} is defined by the following series.  
%  \[
%    \begin{aligned}
%      &\zeta(k_1, \cdots, k_n) := \sum_{0<m_1<\cdots<m_n} {\frac{1} { {m_1}^{k_1} \cdots {m_n}^{k_n} } } 
%    \end{aligned}
%  \]
%For an admissible index $(k_1, \cdots, k_n)$, $k_1+ \cdots +k_n$ is called {\bf weight}, and $n$ is called {\bf depth}.
%\end{defn}

%\subsection{Reformulation of Euler's formulae}
%\label{reform}
%We briefly review the extended double shuffle relations and then we give a reformulation of Euler's formulae \eqref{secE} and \eqref{terE} in Theorem \ref{thma}. 
%We have following relations for double zeta values.
%product

\begin{prop}
\label{propa}
Double zeta values enjoy the {\bf double shuffle relations}. Namely, 
%the following relations 
the {\bf shuffle product formula}  
%{\it Shuffle product formula}:
\begin{align} 
  \label{sh}
  &\zeta(m)\zeta(n)=\sum^{n-1}_{k=0}\binom{m+k-1}{k}\zeta(n-k, m+k)+\sum^{m-1}_{l=0}\binom{n+l-1}{l}\zeta(m-l, n+l),
\end{align}  
%{\it Harmonic product formula}:
and the {\bf harmonic product formula}
\begin{align}  
  \label{ha}
  &\zeta(m)\zeta(n)=\zeta(m,n)+\zeta(n,m)+\zeta(m+n)
\end{align}
hold for $m, n \in \mathbb{Z}_{>1}$.
\end{prop}
%\begin{prop}
%The following relation is the {\bf sum formula} \eqref{sum} for double zeta values. For $k\in\mathbb{Z}_{>2}$,
%\begin{align}   
%  \label{sum}
%  &\sum^{k-1}_{i=2}\zeta(k-i,i)=\zeta(k)
%\end{align}
%\end{prop}

We recall two notions of regularization of double zeta values along the line of \cite{IKZ}\footnote{Another reformulation was also given in \cite{R}.}.
Let $m,n\in\mathbb{Z}_{>0}$. It is shown that when $N\rightarrow\infty$, we have 
\begin{align*}
    	\sum_{0<k<N}\frac{1}{k^m}\sim a_0+a_1(\log N+\gamma),
      \sum_{0<k_1<k_2<N}\frac{1}{{k_1}^m{k_2}^n}\sim b_0+b_1(\log N + \gamma)+b_2(\log N + \gamma)^2,
\end{align*}
with some $a_i, b_i \in \mathbb{R}$ and the Euler's constant $\gamma:=\lim_{n\rightarrow\infty}\big(1+\frac{1}{2}+\cdots+\frac{1}{n}-\log n\big)$. Here $f(x)\sim g(x)\ (x\rightarrow \alpha)$ means $\frac{f(x)}{g(x)}\rightarrow 1\ (x\rightarrow \alpha)$ (this $\alpha$ can be infinity). The {\bf harmonic regularized values} $\zeta_{\mbox{*}}(m)$ and $\zeta_{\mbox{*}}(m,n)$ are defined as follows:
%(i) When $N\rightarrow \infty$,
%  	    \[
%    		\sum_{0<k<N}\frac{1}{k^m}\sim a_0+a_1\cdot(\log N+\gamma)
%           \]
%  	    where $a_0, a_1\in\mathbb{R}$.
%  	    Then $\zeta^*(m)$ and $\zeta^*(m,n)$ are defined by the following polynomials.
           \[
             \zeta_{\mbox{*}}(m):=a_0+a_1\cdot T, \ \zeta_{\mbox{*}}(m,n):=b_0+b_1\cdot T+b_2\cdot T^2\in\mathbb{R}[T].
	     \]   
%(ii) When $N\rightarrow \infty$,
%          \[
%            \sum_{0<k_1<k_2<N}\frac{1}{{k_1}^m{k_2}^n}\sim b_0+b_1\cdot(\log N + \gamma)+b_2\cdot(\log N + \gamma)^2
%          \]
%          where $b_0, b_1, b_2\in\mathbb{R}$.
%          Then $\zeta^{*}(m, n)$ is defined by the following polynomial. 
%          \[
%            \zeta^{*}(m, n):=b_0+b_1\cdot T+b_2\cdot T^2\in\mathbb{R}[T] 
%          \]

In contrast, for $m, n\in\mathbb{Z}_{>0}$, it is also shown that when $\epsilon\rightarrow 0$, we have
\begin{align*}
  	  \sum_{0<k}\frac{ (1-\epsilon)^{k} }{ {k}^m}\sim c_0+c_1(-\log\epsilon),
	   \sum_{0<k_1<k_2}\frac{ (1-\epsilon)^{k_2} }{ {k_1}^m{k_2}^n}\sim d_0+d_1(-\log\epsilon)+d_2(-\log\epsilon)^2,
\end{align*}
with some $c_i, d_i\in\mathbb{R}$. For $m,n\in\mathbb{Z}_{>0}$, the {\bf shuffle regularized values} $\zeta_{\scalebox{0.5}{$\Sha$}}(m)$ and $\zeta_{\scalebox{0.5}{$\Sha$}}(m,n)$ are defined as follows:
%(i) When $\epsilon\rightarrow 0$,
%	\[
%  	  \sum_{0<k}\frac{ (1-\epsilon)^{k} }{ {k}^m}\sim c_0+c_1(-\log\epsilon)
%	\]
%	where $c_0, c_1\in\mathbb{R}$. Then $\zeta^{\Sha}(m)$ is defined by the following polynomial. 
	\[
  	  \zeta_{\scalebox{0.5}{$\Sha$}}(m):=c_0+c_1\cdot T, \  \zeta_{\scalebox{0.5}{$\Sha$}}(m,n):=d_0+d_1\cdot T+d_2\cdot T^2\in\mathbb{R}[T]. 
	\]
%\[
%  \sum_{0<k_1<k_2<N}\frac{1}{{k_1}^m{k_2}^n}\sim a_0+a_1\cdot\log \epsilon+a_2\cdot(\log \epsilon)^2.
%\]
%(ii) When $\epsilon\rightarrow 0$,	
%	\[
%  	  \sum_{0<k_1<k_2}\frac{ (1-\epsilon)^{k^2} }{ {k_1}^m{k_2}^n}\sim d_0+d_1(-\log\epsilon)+d_2(-\log\epsilon)^2
%	\]
% 	where $d_0, d_1, d_2\in\mathbb{R}$. Then $\zeta^{\Sha}(m,n)$ is defined by the following polynomial. 
%	\[
%  	  \zeta^{\Sha}(m,n):=d_0+d_1\cdot T+d_2\cdot T^2 \in \mathbb{R}[T]
%	\]
%\begin{defn}
%There are two regularization for double zeta values. That is, the {\bf series regularized value} $\zeta^*(m,n)$ and the {\bf integral regularized value} $\zeta^\Sha(m,n)$, which are respectively defined as follows. 
%\begin{align*}
%&\zeta^*(m,n):=Z^*(x^{m-1}yx^{n-1}y)\\
%&\zeta^\Sha(m,n):=Z^\Sha(x^{m-1}yx^{n-1}y)
%\end{align*}
%\end{defn}
We remark that
\begin{equation}
\label{remreg}
\zeta_{\mbox{*}}(m)
        =\begin{cases} 
      \zeta_{\scalebox{0.5}{$\Sha$}}(m)=T & \text{if $m=1$},\\
      \zeta_{\scalebox{0.5}{$\Sha$}}(m)=\zeta(m) & \text{if $m>1$},
          \end{cases}
\end{equation}
and $\zeta_{\mbox{*}}(m,n)=\zeta_{\scalebox{0.5}{$\Sha$}}(m,n)=\zeta(m,n)$ if $n>1$.
\begin{prop}
The {\bf extended double shuffle relations} hold for the regularized values. Namely, for $m,n\in\mathbb{Z}_{>0}$, we have
%The equation \eqref{sh} also holds for $\zeta^{\Sha}(m)$ and $\zeta^{\Sha}(m,n)$, and the equation \eqref{ha} also holds for $\zeta\mbox{{\rm *}}(m)$ and $\zeta\mbox{{\rm *}}(m,n)$
\begin{align} 
  \label{regsh}&\zeta_{\scalebox{0.5}{$\Sha$}}(m)\zeta_{\scalebox{0.5}{$\Sha$}}(n)=\sum^{m-1}_{i=0}\binom{n+i-1}{i}\zeta_{\scalebox{0.5}{$\Sha$}}(m-i, n+i)+\sum^{n-1}_{j=0}\binom{m+j-1}{j}\zeta_{\scalebox{0.5}{$\Sha$}}(n-j, m+j),
\end{align}  
\begin{align}  
  \label{regha}
  &\zeta_{\mbox{{\rm *}}}(m)\zeta_{\mbox{{\rm *}}}(n)=\zeta_{\mbox{{\rm *}}}(m,n)+\zeta_{\mbox{{\rm *}}}(n,m)+\zeta_{\mbox{\rm *}}(m+n),
\end{align}
%with $m, n \in \mathbb{Z}_{>0}$.
%two regularized values are connected by regularization relation
%There is another type of relation, called the extended double shuffle relation, which connects two regularizations: 
\begin{align}
\label{edsr}
 &\sum^{m-1}_{i=0}\binom{n+i-1}{i}\zeta_{\scalebox{0.5}{$\Sha$}}(m-i,n+i)+\sum^{n-1}_{j=0}\binom{m+j-1}{j}\zeta_{\scalebox{0.5}{$\Sha$}}(n-j,m+j)\\
 &\quad =\zeta_{\mbox{\rm *}}(m,n)+\zeta_{\mbox{\rm *}}(n,m)+\zeta_{\mbox{\rm *}}(m+n).\nonumber
\end{align}
\end{prop}
%with $m\in\mathbb{Z}_{>1}$ and $n\in\mathbb{Z}_{>0}$. In \cite{IKZ} Proposition 9, it was shown that \eqref{edsr} implies the {\bf sum formula} of Granville (\cite{Gr}) and Zagier (unpublished).
%\begin{equation}
%\label{sum} 
%\sum^{k-1}_{i=2}\zeta(k-i,i)=\zeta(k)
%\end{equation}
%with $k\in\mathbb{Z}_{>2}$.
%The following formula was shown by the regularization and the double shuffle relation. 
%\begin{prop}
%The extended double shuffle relation implies the sum formula. 
%%For k $\in\mathbb{Z}_{>2}$, we have the {\bf sum formula} as follows: 
%%\begin{equation}
%%\label{sum}    
%%\sum^{k-1}_{i=2}\zeta(k-i,i)=\zeta(k).
%%\end{equation}
%\end{prop}
%\begin{proof}
% By using the relation \eqref{edsr} for $m=1$ and $n=k-1$, 
%\begin{equation}
%  \zeta^{\Sha}(1, k-1)+\sum^{k-2}_{j=0}\binom{j}{j}\zeta^{\Sha}(k-1-j, 1+j)=\zeta(k)+\zeta^{\Sha}(1, k-1)+\zeta^{\Sha}(k-1,1).\nonumber
%\end{equation}
%When $n>1$, $\zeta^{\Sha}(m,n)=\zeta(m,n)$. So we can transform it into the following:
%\[
%\sum^{k-2}_{j=1}\zeta(k-1-j, 1+j)=\zeta(k).\nonumber
%\]
%\end{proof}

\label{pr}
%Now we give a proof of Theorem \ref{thma}. 
We recall generating functions of double zeta values and Riemann zeta values which were introduced in \cite{GKZ}. 
For $k\in\mathbb{Z}_{>1}$, we put
\[
D_{k}(X, Y):=\sum_{i=1}^{k-1}\zeta_{\scalebox{0.5}{$\Sha$}}(k-i, i)X^{i-1}Y^{k-i-1}, \
Q_{k}(X,Y):=\sum_{i=1}^{k-1}\zeta_{\scalebox{0.5}{$\Sha$}}(i)\zeta_{\scalebox{0.5}{$\Sha$}}(k-i)X^{i-1}Y^{k-i-1}.
\]
They showed in \cite{GKZ} p.80 (25) the following relation 
\begin{align}
\label{gkz}
D_{k}(X+Y, Y)+D_{k}(X+Y, X)=Q_{k}(X,Y), 
\end{align}
which is a reformulation of shuffle product formula \eqref{regsh}. 
%\begin{prop}
%\label{propgkz}
%$
%D_{k}(X+Y, Y)+D_{k}(X+Y, X)=Q_{k}(X,Y) \label{gkz}
%$
%\end{prop}

The following is a reformulation of Euler's problematic key formula \eqref{key}. 
%Its proof will be provided in \S \ref{pr}. 
%\begin{rem}
%For shuffle product \eqref{sh} and harmonic product \eqref{ha}, the following hold.
%\begin{align*} 
%  &\zeta^\Sha(m)\zeta^\Sha(n)=\sum^{n-1}_{k=0}\binom{m+k-1}{k}\zeta^\Sha(n-k, m+k)+\sum^{m-1}_{l=0}\binom{n+l-1}{l}\zeta^\Sha(m-l, n+l),\\
%&\zeta^*(m)\zeta^*(n)=\zeta^*(m,n)+\zeta^*(n,m)+\zeta^*(m+n).
%\end{align*}  
%
%\end{rem}
%methodus
\begin{thm}
\label{thma}
For $m,n\in \mathbb{Z}_{>1}$,    
\begin{align}
\label{thm7} 
\zeta(m,n)=P(m,n),
\end{align}
where
  \begin{align}
   \label{P(m,n)}
    P(m,n):=&\sum^{m-1}_{i=0}(-1)^{i}\binom{n+i-1}{i}\zeta(n+i)\zeta_{\scalebox{0.5}{$\Sha$}}(m-i)\\
               &\ +(-1)^{m}\sum^{n-1}_{j=0}\binom{m+j-1}{j}\Bigl\{\zeta(m+j)\zeta_{\scalebox{0.5}{$\Sha$}}(n-j)-\zeta^{\star}(n-j,m+j)\Bigr\}\nonumber.
  \end{align}                 
%  \begin{align}
%    \label{eqsec} 
%    &\zeta(m)\zeta(n)-\zeta(m+n)=P(m,n)+P(n,m),
%  \end{align} 
%and
%  \begin{align} 
%   \label{eqter}
%    \zeta(m, n)=&\sum^{m-1}_{i=0}(-1)^{i}\binom{n+i-1}{i}\Bigl\{\zeta_{\mbox{{\rm *}}}(n+i, m-i)+\zeta(m-i, n+i)+\zeta(m+n)\Bigr\}\\ 
%                    &+(-1)^{m}\sum^{n-1}_{j=0}\binom{m+j-1}{j}\zeta_{\mbox{{\rm *}}}(m+j, n-j),\nonumber
%  \end{align}
%where
%  \begin{align}
%   \label{P(m,n)}
%    P(m,n):=&\sum^{m-1}_{i=0}(-1)^{i}\binom{n+i-1}{i}\zeta(n+i)\zeta_{\scalebox{0.5}{$\Sha$}}(m-i)\\
%               &\ +(-1)^{m}\sum^{n-1}_{j=0}\binom{m+j-1}{j}\Bigl\{\zeta(m+j)\zeta_{\scalebox{0.5}{$\Sha$}}(n-j)-\zeta^{\star}(n-j,m+j)\Bigr\}\nonumber.
%  \end{align}                 
%Furthermore, these relation can be derived from the extended double shuffle relations \eqref{regsh}, \eqref{regha} and \eqref{edsr}. 
\end{thm}

%Now we give a proof of Theorem \ref{thma}.
\begin{proof}
By changing a variable $X$ to $X-Y$ and putting $k=m+n$ in \eqref{gkz}, we get
\begin{align}
  D_{m+n}(X, Y)+D_{m+n}(X, X-Y)=Q_{m+n}(X-Y, Y). \label{gkza}
\end{align}
We note that
\[
\sum^{m-1}_{i=0}(-1)^{i}\binom{n+i-1}{i}\zeta(n+i)\zeta_{\scalebox{0.5}{$\Sha$}}(m-i)
\] 
is the coefficient of $X^{n-1}Y^{m-1}$ in $Q_{m+n}(X-Y, Y)$, while
\[
(-1)^{m}\sum^{n-1}_{j=0}\binom{m+j-1}{j}\Bigl\{\zeta(m+j)\zeta_{\scalebox{0.5}{$\Sha$}}(n-j)-\zeta^{\star}(n-j,m+j)\Bigr\}
\] 
is the coefficient of $X^{n-1}Y^{m-1}$ in $-D_{m+n}(X, X-Y)$ and $\zeta(m,n)$ is the coefficient of $X^{n-1}Y^{m-1}$ in $D_{m+n}(X,Y)$. Therefore it follows from \eqref{gkza} that $P(m,n)=\zeta(m,n)$.
\end{proof}

The following theorem gives a reformulation of Euler's secunda methodus \eqref{secE} and tertia methodus \eqref{terE}. 
\begin{thm}
\label{thmb}
For $m,n\in \mathbb{Z}_{>1}$, we have
  \begin{align}
    \label{eqsec} 
    &\zeta(m)\zeta(n)-\zeta(m+n)=P(m,n)+P(n,m),
  \end{align} 
and
  \begin{align} 
   \label{eqter}
    \zeta(m, n)=&\sum^{m-1}_{i=0}(-1)^{i}\binom{n+i-1}{i}\Bigl\{\zeta_{\mbox{{\rm *}}}(n+i, m-i)+\zeta(m-i, n+i)+\zeta(m+n)\Bigr\}\\ 
                    &+(-1)^{m}\sum^{n-1}_{j=0}\binom{m+j-1}{j}\zeta_{\mbox{{\rm *}}}(m+j, n-j),\nonumber
  \end{align}
\end{thm}
\begin{proof}
\leavevmode \\ {\it Secunda methodus}: By Theorem \ref{thma}, $P(m, n) = \zeta(m, n)$ for $m, n\in \mathbb{Z}_{>1}$. So by using the harmonic product formula \eqref{ha}, we obtain
\begin{align*}
P(m, n)+P(n, m)&=\zeta(m, n)+\zeta(n, m)
                 %\intertext{by the harmonic product formula \eqref{ha}, }
                 =\zeta(m)\zeta(n)-\zeta(m+n). 
\end{align*}
Hence \eqref{eqsec} is shown. \\
{\it Tertia methodus}: Again by Theorem \ref{thma}, the extended double shuffle relations, \eqref{regsh}, \eqref{regha} and \eqref{edsr}, yield 
\begin{align*}
	&\zeta(m, n)=P(m, n)\\
	&=\sum^{m-1}_{i=0}(-1)^{i}\binom{n+i-1}{i}\zeta(n+i)\zeta_{\scalebox{0.5}{$\Sha$}}(m-i)+(-1)^{m}\sum^{n-1}_{j=0}\binom{m+j-1}{j}\zeta_{\scalebox{0.5}{$\Sha$}}(m+j, n-j),\\
      &=\sum^{m-1}_{i=0}(-1)^{i}\binom{n+i-1}{i}\zeta(n+i)\zeta_{\scalebox{0.5}{$\Sha$}}(m-i)+(-1)^{m}\sum^{n-1}_{j=0}\binom{m+j-1}{j}\Bigl\{\zeta(m+n-1)T\\
      &\quad -\zeta(1, n+m-1)-\sum^{m+n-2}_{k=1}\zeta(m+n-1-k, 1+k)\Bigr\},\\	
\intertext{by using the equation \eqref{edsr},}      
&=\sum^{m-1}_{i=0}(-1)^{i}\binom{n+i-1}{i}\zeta(n+i)\zeta_{\scalebox{0.5}{$\Sha$}}(m-i)+(-1)^{m}\sum^{n-1}_{j=0}\binom{m+j-1}{j}\Bigl\{\zeta(m+n-1)T\\
      &\quad -\zeta(1,n+m-1)-\zeta(m+n)-\zeta_{\scalebox{0.5}{$\Sha$}}(1, n+m-1)%+\zeta_{\scalebox{0.5}{$\Sha$}}(1, m+n-1)\Bigr\}
\\
      &\quad+\sum^0_{j=0}\binom{m+n-2+j}{j}\zeta_{\scalebox{0.5}{$\Sha$}}(1-j, m+n-1+j)\Bigr\}\\	
&=\sum^{m-1}_{i=0}(-1)^{i}\binom{n+i-1}{i}\zeta(n+i)\zeta_{\scalebox{0.5}{$\Sha$}}(m-i)+(-1)^{m}\sum^{n-1}_{j=0}\binom{m+j-1}{j}\zeta_{\mbox{*}}(m+j, n-j),\\
\intertext{by Remark \ref{remreg},}
	&=\sum^{m-1}_{i=0}(-1)^{i}\binom{n+i-1}{i}\zeta(n+i)\zeta_{\mbox{*}}(m-i)+(-1)^{m}\sum^{n-1}_{j=0}\binom{m+j-1}{j}\zeta_{\mbox{*}}(m+j, n-j)\\
	\intertext{and finally by the harmonic product formula \eqref{regha},}
%       &=\sum^{m-2}_{i=0}(-1)^{i}\binom{n+i-1}{i}\zeta(n+i)\zeta(m-i) + (-1)^{m}\sum^{n-2}_{j=0}\binom{m+j-1}{j}\zeta(m+j, n-j)\\
%	&\quad +(-1)^{m-1}\binom{m+n-2}{m-1}\zeta(m+n-1)\zeta^\Sha(1)+(-1)^m\binom{m+n-2}{n-1}\zeta^\Sha(m+n-1,1)\\
%	&=\sum^{m-2}_{i=0}(-1)^{i}\binom{n+i-1}{i}\zeta(n+i)\zeta(m-i)\\
%      &\quad+ (-1)^{m}\sum^{n-2}_{j=0}\binom{m+j-1}{j}\zeta(m+j, n-j)\\
%	&\quad+(-1)^{n-1}\binom{m+n-2}{m-1}\zeta(m+n-1)\zeta^\Sha(1)+(-1)^m\binom{m+n-2}{n-1}\Bigl\{ \zeta(m+n-1)\zeta^\Sha(1)\\
%      &\quad-\zeta(1,m+n-1)-\sum^{m+n-2}_{l=1}\zeta(m+n-1-l,1+l)\Bigr\}\\
%       \intertext{by harmonic product \eqref{ha}, }
%	&=\sum^{m-2}_{i=0}(-1)^{i}\binom{n+i-1}{i}\Bigl\{\zeta(n+i,m-i)+\zeta(m-i, n+i)+\zeta(m+n)\Bigr\}\\
%      &\quad+ (-1)^{m}\sum^{n-2}_{j=0}\binom{m+j-1}{j}\zeta(m+j, n-j)\\
%	&\quad+(-1)^m\binom{m+n-2}{n-1}\Bigl\{-\zeta(1,m+n-1)-\sum^{m+n-2}_{l=1}\zeta(m+n-1-l,1+l)\Bigr\}\\
%      \intertext{by sum formula \eqref{sum}, }
%	&=\sum^{m-2}_{i=0}(-1)^{i}\binom{n+i-1}{i}\Bigl\{\zeta(n+i,m-i)+\zeta(m-i, n+i)+\zeta(m+n)\Bigr\}\\
%      &\quad+ (-1)^{m}\sum^{n-2}_{j=0}\binom{m+j-1}{j}\zeta(m+j, n-j)\\
%	&\quad+(-1)^m\binom{m+n-2}{n-1}\Bigl\{-\zeta(1,m+n-1)-\zeta(m+n)\Bigr\}\\
%	&=\sum^{m-2}_{i=0}(-1)^{i}\binom{n+i-1}{i}\Bigl\{\zeta(n+i,m-i)+\zeta(m-i, n+i)+\zeta(m+n)\Bigr\}\\
%      &\quad+(-1)^{m}\sum^{n-2}_{j=0}\binom{m+j-1}{j}\zeta(m+j, n-j)\\
%	&\quad+(-1)^m\binom{m+n-2}{n-1}\zeta^\Sha(m+n-1,1)\\
%	&\quad+(-1)^m\binom{m+n-2}{n-1}\Bigl\{-\zeta^\Sha(m+n-1,1)-\zeta(1,m+n-1)-\zeta(m+n)\Bigr\}\\
	&=\sum^{m-1}_{i=0}(-1)^{i}\binom{n+i-1}{i}\Bigl\{\zeta_{\mbox{*}}(n+i, m-i)+\zeta(m-i, n+i)+\zeta(m+n)\Bigr\}\\ 
    &\quad+(-1)^{m}\sum^{n-1}_{j=0}\binom{m+j-1}{j}\zeta_{\mbox{*}}(m+j, n-j).
\end{align*}
Hence \eqref{eqter} is shown.
\end{proof}
%\subsection{Proof of Theorem \ref{thma}}
\label{pr}

\section*{Acknowledgments}
The author is deeply grateful to Professor H. Furusho for guiding him towards this topic. This paper could not have been written without his continuous encouragements. He gratefully acknowledges the referee for indicating him a very clear proof with idea of using the generating functions $D_k(X,Y)$ and $Q_k(X,Y)$ in \cite{GKZ} which greatly improved this paper. He would also like to thank H. Bachmann for giving him some comments on this paper. 
%\clearpage


\begin{thebibliography}{Utah}
\bibitem{Eu} L. Euler, {\it Meditationes circa singulare serierum genus}, Novi Comm. Acad. Sci. Petropol {\bf 20} (1776), 140--186, reprinted in Opera Omnia ser. I, vol. 15, B. G. Teubner, Berlin (1927) 217--267. 

\bibitem{GKZ} H. Gangl, M. Kaneko, and D. Zagier, {\it Double zeta values and modular forms}, Automorphic forms and zeta functions, 71--106, World Sci. Publ., (2006).

\bibitem{Gr} A. Granville, {\it A decomposition of Riemann's zeta-function}, in London Math. Soc. Lecture Note Ser. {\bf 247}, Cambridge, 1997, pp.95--101.  

%
%\bibitem{H} M. Hoffman, {\it The algebra of multiple harmonic series}, J. of Algebra, {\bf194} (1997), 477--495. 
\bibitem{IKZ} K. Ihara, M. Kaneko, and D. Zagier, {\it Derivation and double shuffle relations for multiple zeta values}, Compositio Math. {\bf 142} (2006), 307--338.

\bibitem{R} G. Racinet, {\it Doubles m\'elanges des polylogarithmes multiples aux racines de l'unit\'e}, Publ. Math. Inst. Hautes Etudes Sci.  {\bf 95} (2002), 185--231. 
%\bibitem{YKH} Y. Komori, K. Matsumoto, H. Tsumura, {\it Shuffle products for multiple zeta values and partial fraction decompositions of zeta-functions of root systems}, Mathematische Zeitschrift (2011), Volume {\bf 268}, Issue 3, 993--1011.

\end{thebibliography}
\end{document}